\documentclass[12pt,a4paper,fleqn]{article}
\usepackage{amsmath,amssymb,euscript,amsthm,amsfonts}

\usepackage{exscale,amsthm,amssymb}
\usepackage{epsfig, graphicx}
\usepackage[english]{babel}
\usepackage{array}

\newtheorem{definition}{Definition}
\newtheorem{umova}{Assumption}
\newtheorem{theorem}{Theorem}
\newtheorem{lemma}{Lemma}

\theoremstyle{definition}
\theoremstyle{remark}

\newcommand{\bol}[1]{\mbox{\boldmath$#1$}}
\newcommand{\bmu}{\bol{\mu}}
\newcommand{\bxi}{\bol{\xi}}
\newcommand{\balpha}{\bol{\alpha}}
\newcommand{\bbeta}{\bol{\beta}}
\newcommand{\bgamma}{\bol{\gamma}}
\newcommand{\by}{\mathbf{y}}
\newcommand{\bx}{\mathbf{x}}
\newcommand{\br}{\mathbf{r}}
\newcommand{\bn}{\mathbf{n}}
\newcommand{\bk}{\mathbf{k}}
\newcommand{\cE}{{\mathcal{E}}}
\newcommand{\cN}{{\mathcal{N}}}
\newcommand{\cP}{{\mathcal{P}}}
\newcommand{\natvec}{{\cN}_{0}^\infty}
\newcommand{\mE}{\mathrm{E}}

\newcommand{\ba}{\mathbf{a}}

\newcommand{\pxi}{\bxi_{\bx t}(X)}
\newcommand{\pmu}{\bmu_{\bx t}(X)}
\newcommand{\cA}{\mathcal{A}}

\begin{document}
%\renewcommand{\proofname}{Proof}
%\makeatletter
%\headsep 10 mm
%\footskip 10 mm
    \title{Asymptotic behaviour of the $S$-stopped branching processes with countable state space}

    \author{Iryna Kyrychynska\thanks{Chair of Theoretical and Applied Statistics, Department of Mechanics and Mathematics, Ivan Franko National University of Lviv, Universytetska Str. 1, Lviv, 79000, Ukraine.}, Ostap Okhrin\thanks{C.A.S.E. - Center for Applied Statistics and Economics, Ladislaus von Bortkiewicz Chair of Statistics of Humboldt-Universit{\"a}t zu Berlin, Spandauer Stra{\ss}e $1$, D-$10178$ Berlin, Germany. Email: ostap.okhrin@wiwi.hu-berlin.de.}, Yaroslav Yeleyko\thanks{Chair of Theoretical and Applied Statistics, Department of Mechanics and Mathematics, Ivan Franko National University of Lviv, Universytetska Str. 1, Lviv, 79000, Ukraine.}}
% \date{\today}

    \maketitle
    
    \begin{center}
        This paper was published in \emph{Visn. L'viv. Univ., Ser. Mekh.-Mat.} (Bulletin of the
        Lviv University, Series in Mechanics and Mathematics) Vol.67, pp.119-129 (2007).

        Reviewed in \emph{Zentralblatt f\"ur Mathematik} Zbl 1164.60418.\\[2cm]
    \end{center}

{\small \noindent \textbf{Abstract:}
    The starting process with countable number of types $\mu(t)$
    generates a stopped branching process $\xi(t)$. The starting process stops, by falling into the nonempty set
    $S$. It is assumed, that the starting process is subcritical, indecomposable and noncyclic.
    It is proved, that the extinction probability converges to the cyclic function with period 1.
}

\vspace*{0.5cm} {\small \textbf{Keywords}: branching processes; Markov chain; extinction
probability; asymptotic behavior.
}

\textbf{1. }Let us consider a measured state space $(X,\cA)$, where $\cA$ is the $\sigma$-algebra on
$X$. On this space we consider unbreakable homogenous Markov process with transition probability $P(t,x,A)$, where $t$ denotes time, $x \in X$ and $A\in \cA$. Considering every trajectory of the given process as an evolution of the motion of a particle, $P(t,x,A)$ can be interpreted as a probability that a particle, which starts its motion from $x \in X$, falls into the set $A\in \cA$ till the time $t$. It is assumed, that the time is discrete and the
lifetime of a particle is equal to $1$. At the end of its life the particle promptly gives rise to a number of offsprings, starting position of which are randomly distributed on the space $X$. The number and the position of these offsprings depends only on the position of the particle-ancestor at the transformation time point. Further every offspring evolutes analogously and independently of other particles.

Let $\mu_{xt}(A)$ be a random measure, which for every $A \in \cA$ is equal to the number of the particles at time point $t$, types of which fall into set $A$, under condition that the process started with one particle $x\in X$. $\mu_{t}(A)$ is a random measure equal to the number of particles at the time $t$, which types are from the set $A$, but without any knowledge about starting group of particles.

Further we assume, that the space $X$ consists of a countable number of elements
$x_1,\ldots,x_n,\ldots$. This means that the set of types of particles $\{T_1,\ldots,T_n,\ldots\}$
is countable.

Based on the measure $\mu_{xt}(A)$ we introduce the multivariate measure $\bmu_{\bx t}(A)$
\begin{equation*}
    \bmu_{\bx t}(A)=\left\{
        \begin{array}{cl}
        \sum \limits_{i=0}^\infty \sum \limits_{j=1}^{n_i}\mu_{x_{ij} t}(x_m), & \mbox{if } x_m \in A \\
        0, & \mbox{else}
        \end{array}
    \right\}_{m=0}^\infty,
\end{equation*}

where $\bx = \{x_{11},\ldots,x_{1n_1},x_{21},\ldots, x_{2n_2},\ldots\},\ x_{ij}\in X$ is the
$j$-th element $i$-th type.

Let us denote $\cN_{0}=\{0,1,2,\ldots\}$, and respectively $\natvec$ is an
infinite dimensional measured space with elements $x_i \in \cN_{0}$.

Having $P(t,x,A)$ let us introduce $\widehat{P}(t,\bx,\by),\, (\bx,\by \in
\natvec)$, where $\widehat{P}$ is a probability that we obtain vector $\by$ till
time $t$, assuming that we started from $\bx$. $\widehat{P}$ could be rewritten in terms of $\bmu_{\bx t}$ as
\[\widehat{P}(t,\bx,\by) = P\{\pmu=\by\}.\]
Let $\cE(i)=(\delta_{i1},\ldots,\delta_{in},\ldots),$ where $\delta_{ij}$ is the Kronecker
symbol, $\delta_{\bx\by}=\prod_{i=1}^\infty\delta_{x_iy_i}$, $\cE(i)$ is the
particle of the $i$-th type. We also assume, that $a^b = a_1^{b_1}a_2^{b_2}\cdots a_n^{b_n}\cdots,\; a! = a_1!a_2!\cdots a_n!\cdots,\; \overline{a} = a_1+\cdots+a_n+\cdots,\; a_i^{[b_i]} = a_i(a_i-1)\cdots(a_i-b_i+1)$.
\begin{definition}\label{ftvirnyj}
    Functional
    \[F(s(\cdot))=F(s)=\mE \exp{\left\{\int \ln{s(x)}\mu(dx)\right\}}\]
    is called a generated functional of the random measure $\mu$, where $s(x)$ is a measured bounded function.
\end{definition}
Generated functional $F(s)$ is always defined, when $0<|s(x)|\leq 1$ and integral $\int
\ln{s(x)}\,\mu(dx)$ exists.

For our process the generated functional is given as
\begin{equation*}
    h(t,s(\cdot))=\mE\exp{\left\{\int_X \ln{s(z)}\bmu_{t}(dz)\right\}},
\end{equation*}
where $\bmu_{t}$ is the same multivariate measure as $\bmu_{\bx t}$ but not taking into account any combination of the starting position of the process. Further we will consider the case $s(\cdot) = const = s = (s_1,s_2,\ldots)$. It is easy to check
whether the introduced generated functional is generated, as in the case of finite number of types (in that case it is not a functional but a function).

Let us denote
\begin{eqnarray*}
    h^i(t,s) &=& h^{\cE(i)}(t,s), \\
    h^{\bbeta}(t,s) &=& \big((h^ {\cE(1)}(t,s))^{\bbeta_1}, (h^
    {\cE(2)}(t,s))^{\bbeta_2},\ldots\big), \\
    h(t,s) &=& \big(h^{\cE(1)}(t,s), h^ {\cE(2)}(t,s),\ldots\big).
\end{eqnarray*}
It is proved in [3], that the introduced generated function follows the main functional
equation ($\forall \; t,\tau = 0,1,2,\ldots$)
\begin{equation*}
    h(t+\tau,s)=h(t,h(\tau,s)).
\end{equation*}
Let us fix the finite subset $S \subset \natvec,\, 0\notin S$. {\it Stopped} or
{\it$S$-stopped} multitype branching process is the process $\pxi$, defined for $t=1,2,\ldots$ and $\bx\in\natvec$ by equations
\begin{eqnarray*}
    \pxi=\left\{
    \begin{array}{rl}
        \pmu, &\; \mbox{if } \forall v, \; 0\leq v < t, \; \bmu_{\bx v}(X) \notin S \\
        \bmu_{\bx u}(X), &\; \mbox{if } \forall v, \; 0\leq v < u, \;
        \bmu_{\bx v}(X) \notin S, \; \bmu_{\bx u}(X) \in S, \; u<t.
    \end{array}
    \right.
\end{eqnarray*}
From this, for the $S$-stopped process $\pxi$, points of the
set $S$ are additional states of absorption compared to the process $\bmu_{\bx
t}(X)$. The latter had only one point of absorption $0$. In contrast to the process ${\boldsymbol
\mu}_{\bx t}(X)$, in the $S$-stopped branching process $\pxi$
single particles in generation $t$ multiplies independently following probability law defined by the generated functional $h(\cdot)$, only if $\pxi \notin
S$. If the random vector $\pxi$ falls into the set $S$, the evolution
of the process stops.

Since the process $\pmu$ is a Markov chain, then
\begin{equation*}
    \widehat{P}(t_1+t_2,\balpha,\bbeta)=\sum_{\bgamma\in\natvec}\widehat{P}(t_1,\balpha,\bgamma)\widehat{P}(t_2,\bgamma,\bbeta).
\end{equation*}
For further needs, we also consider probabilities $\widetilde{P}(t,\balpha,\br)$, defined as
\begin{eqnarray}\label{ptilde}
    \widetilde{P}(t,\balpha,\br)=\left\{
    \begin{array}{ll}
        \widehat{P}(1,\balpha,\br),&t=1;\\
        \sum_{\bbeta\notin
        S}\widehat{P}(1,\balpha,\bbeta)\widetilde{P}(t-1,\bbeta,\br),&t\geq2.
    \end{array}
    \right.
\end{eqnarray}
It is easy to see, that $\widetilde{P}(l,\balpha,\br)$ is a conditional
probability of the event
\begin{equation*}
    \big\{\bmu_{\balpha l}(X)=\br\big\}\bigcap \left(\bigcap_{l'=1}^{l-1}\big\{
\bmu_{\balpha l'}(X)\notin S\big\}\right).
\end{equation*}
Let
\begin{equation*}
    q^\bn_\br(t)=P\big\{ \bxi_{\bn t}(X) = \br\big\}
\end{equation*}
be the probability of an extinction of the $S$-stopped branching process $\bxi_{\bx t}(X)$ into state $\br\in S$ till time $t$, starting from state $\bn \in {\cN_0^\infty}$.

\textbf{2. Main facts. }
\begin{theorem}
    For any $\bn\notin S, \; \bn\neq 0,\; \br\in S,\; t\geq 1$ holds
    \begin{equation}
        \label{tma1:3} q^\bn_\br(t)=\sum_{\balpha\in
        S}\sum^t_{l=1}c_{\balpha\br}(t,l)\widehat{P}(l,\bn,\balpha),
    \end{equation}
    where coefficients $c_{\balpha\br}(t,l)$ can be found from
    %labels(3,4,5)
    \begin{eqnarray}
        \label{equ:cdef1} c_{\balpha\br}(t+1,l+1) &=& c_{\balpha\br}(t,l), \\
        \label{equ:cdef2} c_{\balpha\br}(t+1,1) &=& \delta_{\balpha\br}-\sum_{l=1}^{t-1}\widetilde{P}(l,\balpha,\br), \\
        \label{equ:cdef3} c_{\balpha\br}(1,1) &=& \delta_{\balpha\br}.
    \end{eqnarray}
\end{theorem}
\begin{proof}
    Let
    \begin{equation*}
        \tau=\min{\{t:\!{\bmu_{\bn t}(X)}\in S\}}
    \end{equation*}
    be the moment of the first fall into $S$, then for $t\geq l$
    \begin{equation*}
        P\{\bxi_{\bn t}(X)=\br,\ \tau=l\}=P\{\bxi_{\bn l}(X)=\br\}=\widetilde{P}(l,\bn,\br).
    \end{equation*}
    Applying (\ref{ptilde}) to $\widetilde{P}(l,\bn,\br),\ l\geq 2,$ we get
    \begin{eqnarray*}
        \widetilde{P}(l,\bn,\br) &=& \sum_{\balpha\notin
        S}\widehat{P}(1,\bn,\balpha)\widetilde{P}(l-1,\balpha,\br)\\
        &=& \sum_{\balpha\notin
        S}\widehat{P}(2,\bn,\balpha)\widetilde{P}(l-2,\balpha,\br)-\sum_{\balpha\in
        S}\widehat{P}(1,\bn,\balpha)\widetilde{P}(l-1,\balpha,\br).
    \end{eqnarray*}
    The first sum on the right hand side of this formula can be transformed
    similarly
    \begin{equation*}
        \sum_{\balpha\notin
        S}\widehat{P}(2,\bn,\balpha)\widetilde{P}(l-2,\balpha,\br)=\sum_{\balpha\notin
        S}\widehat{P}(3,\bn,\balpha)\widetilde{P}(l-3,\balpha,\br)-\sum_{\balpha\in
        S}\widehat{P}(2,\bn,\balpha)\widetilde{P}(l-2,\balpha,\br).
    \end{equation*}
    Making the same transformations in the sum $\sum_{\balpha\notin
    S}\widehat{P}(i,\bn,\balpha)\widetilde{P}(l-i,\balpha,\br)$, we get
    %labels(6,7)
    %
    \begin{eqnarray}\label{equ_pt:ptildelong}
        \widetilde{P}(l,\bn,\br)&=&\widehat{P}(l,\bn,\br)-\sum_{\balpha\in
        S}\sum_{i=1}^{l-1}\widehat{P}(l-i,\bn,\balpha)\widetilde{P}(i,\balpha,\br),\\\label{equ_pt:ptildelong_help}&\;&l=2,\ldots,t,\;\widetilde{P}(l,\bn,\br)=\widehat{P}(1,\bn,\br).
    \end{eqnarray}
    As $q^\bn_\br(t)=\sum_{l=1}^t\widetilde{P}(l,\bn,\br)$, from formulas (\ref{equ_pt:ptildelong}),(\ref{equ_pt:ptildelong_help})
    we get (\ref{equ:cdef1}),(\ref{equ:cdef2}),(\ref{equ:cdef3}).
\end{proof}
Late on we will consider the process similarly to [1]. Let
\begin{equation*}
    A_1 (x, D) = \mE \{\xi_{x 1}(D) \}
\end{equation*}
be the first factorial moment, where $\xi_{x 1}(D)$ is such a random measure, which for each $D \in \cA$ is equal to the number of particles at time point $1$, which types belong to set $D$, conditional on $S$-stopped process. It also taken into account that at the beginning there was only one particle of the type $x\in X$, what means $\bxi_{\bx 1}(D)=\sum_{i=1}^\infty \xi_{x_1 1}(D)$. From the linearity of $\mE$ we have $A_1(\bx,D)=\mE \{\bxi_{\bx 1}(D)\}=\sum_{i=1}^\infty A_1(x_i,D)$. It is important that $D$ could be a vector or a set.

\begin{definition}
    Let $A_1 (x, D) = A (x, D)$ and
    \begin{equation*}
        A_{n+1} (x, D) = \int_{X} A_n (y, D)\,d A (x, y) = \int_{X} A (y, D)\,d A_n (x, y).
    \end{equation*}
    It is assumed, that $A_0 (x, D) = 1$, if $x\in D$ and $A_0 (x, D) = 0$ else.
\end{definition}
In [4] it is proved, that iterations of the operator $A$ coincide with the first moments of
$\xi$. This means, that for matrix of the linear operator $A(t)$, with $A_{ij}(t)=A_t (x_i, x_j)$, it holds that $A(t)=A^t$ will take place, where $A=A(1)$.

Let
\begin{equation*}
    B_t (x, D_1, D_2) = \mE \{\xi_{x t}(D_1)\cdot \xi_{x t}(D_2) - \xi_{x t}(D_1\cap D_2) \}
\end{equation*}
be the second factorial moment.

For further work we have to introduce some definitions, describing classes of branching processes (see [3]).
\begin{definition}
    Branching process in which all types form a single class of equivalent types is called {\it{indecomposable}}. All other processes are called {\it {decomposable}}. Branching process is called fully {\it{indecomposable}} if the set of types could be split-up into two nonempty closed sets.
\end{definition}
\begin{definition}
    An indecomposable discrete time branching process is called {\it{cyclic}} with period $d$, if the greatest divisor for all $t$, such that $\langle A_t (x_i, x_i)\rangle >0$, is equal to $d$. If $d=1$ then the process is called noncyclic.
\end{definition}
\begin{definition}
    An indecomposable discrete time branching process is called {\it{subcritical}}, if the largest eigenvalue (Perron's root) $\delta$ of the matrix $A$ is smaller than 1, {\it{supercritical}}, if $\delta>1$ and {\it{critical}} if $\delta=1$
    and $f(x_i)B_{jk}^i \nu(x_j)\nu(x_k)>0$, where $B_{jk}^i$ is the matrix of the operator $B$, and $f$ and $\nu$ eigenfunction and invariant measure respectively which correspond do the Perron's root $\delta$.
\end{definition}
\begin{umova}{\label{u1}}
    The kernel $\mE \bxi_{\bx t}(S)$ is assumed to be indecomposable, noncyclic and
    subcritical.
\end{umova}
Correspondingly to the assumption \ref{u1} the operator $A$, which is defined by the kernel $\mE \{\xi_{x t}(D)\}$ in the space of measurable functions and in the space of
measures, has the eigenfunction $f(\cdot)$ and the invariant measure $\nu(\cdot)$, such that
\begin{eqnarray*}
    \int_X f(y) A_t (x, dy)&=&f(x)\;=\;\sum_{i=1}^\infty f(y_i)A_t(x,y_i),\\
    \int_X A_t (x, Y) \nu(dx)&=&\nu(Y)\;=\;\sum_{i=1}^\infty A_t(x_i,Y)\nu(x_i).
\end{eqnarray*}
Further we assume, that $0<x_1<f(x)<x_2<\infty$, $\nu(X)<\infty$ and
\begin{eqnarray}
    \label{bazovi:fnu} \int_X f(y)\nu(dy)=1=\sum_{i=1}^\infty
    f(y_i)\nu(y_i).
\end{eqnarray}
The operator induced by the above defined kernel in the space of bounded functions has $\{1\}$ as an
isolated point of the spectrum.
\begin{umova}\label{u2}
    We assume $\mE \{\bmu_{\cE(j)1}(x_i)
    \log{\bmu_{\cE(j)1}(x_i)}\}$ is finite for $\forall i,j=1,2,\ldots$.
\end{umova}
\begin{umova}\label{u3}
    The expansion $A_t(\bx,\by)=\sum_k f(x_k)\delta_k^t
    \nu(y_k)$ exists.
\end{umova}
As in indecomposable, noncyclic, subcritical processes with discrete time all absolute values of eigenvalues are less than one, then based on the assumption \ref{u3} we can conclude, that when $t\rightarrow\infty$
\begin{equation*}
    A_t(x_i,y_j) = f(x_i)\delta^t \nu(y_j)+o(\delta_1^t),
\end{equation*}
where $\delta$ is the largest eigenvalue. Thus
\begin{equation}
    \label{a:a_t_infty}
    \lim_{t\rightarrow\infty}A_t(x_i,y_j)\delta^{-t}=f(x_i)\nu(y_j).
\end{equation}
Let us denote
\begin{eqnarray*}
    R^i(t,s) &=& 1-h^i(t,s), \\
    R(t,s) &=& (R^1(t,s),\ldots,R^n(t,s),\ldots), \\
    R(t,0) &=& Q(t) = (Q^1(t),\ldots,Q^n(t),\ldots)=
    \lim_{s\rightarrow 0}R(t,s).
\end{eqnarray*}
As in the case with the finite number of types, the  following inequalities could be easily proved (see [3])
%
%druhe inequality label 8
\begin{eqnarray}
    \label{neq_R:neq_r1} 0\leq R^i(t,s)\leq Q^i(t) \; &\mbox{ïðè}& \; 0 <|s|\leq1, \\
    \label{neq_R:neq_r2} |R^i(t,s)|\leq2Q^i(t) \; &\mbox{ïðè}& \;
    0 <|s|\leq1.
\end{eqnarray}
(\ref{neq_R:neq_r2}) implies that for the degenerating branching
processes $R^i(t,s)$ converges uniformly to zero on $0<|s|\leq 1$.

We need following technical assumption on the process
\begin{umova}
    \label{u35} Let $A^t>0$ for some $t>0$ in the sense $\forall i, j \; a_{ij}>0$
    and $h^i(t,s)\neq A_{ij}(t).$
\end{umova}
Hereafter the notation $A=\{a_{ij}\}>0$, means that $a_{ij}>0\;\forall i,j$, and the notation $A>B$, where $A=\{a_{ij} \}, B= \{b_{ij}\}$ are matrices, means that $a_{ij}>b_{ij}\;\forall i,j.$

Let $h(s)=h(1,s)$.
\begin{umova}
    \label{dyvna_umova}
    Following the above defined assumptions for this process, it holds that
    \begin{equation}
        \label{tma2:16} 1-h(s)=[A-E(s)](1-s),
    \end{equation}
    where matrix $E(s)$ with $0\leq s \leq s' \leq 1$ satisfies conditions
    $0\leq E(s')\leq E(s)\leq A$ and $\lim_{s\rightarrow 1}E(s)=0$.
\end{umova}

\begin{theorem}
\label{teorema_vidnosh_r} With Assumptions \ref{u3}-\ref{dyvna_umova}
\begin{equation*}
    \lim_{t\rightarrow\infty}\frac{R^i(t,s)}{f(x_k)R^k(t,s)}=\nu(x_i)
\end{equation*}
uniformly on all $s\neq 1$, $0\leq s\leq1$.
\end{theorem}
This theorem is proved analogically to theorem 1 on page 192 in [3], by replacing the right and left eigenvectors by eigenfunction and invariant measure respectively. Matrices are from the class of matrices of infinite measurable linear operator.
\begin{theorem}
    \label{tmaB} By assumptions \ref{u1}-\ref{dyvna_umova} for any $i,j=1,2,\ldots$ and for
    $l\rightarrow\infty$ probability that the process extinct to $0$ from one particle of type $j$ over $l$ is
    \begin{equation}
        \label{osnovna_tma_8}
        1-\widehat{P}(l,\cE(j),0)=K(S_j)\delta^l(1+o(1)) \mbox{, äå }
        K(S_j)>0;
    \end{equation}
    a) the limit of the conditional probabilities exists
    \begin{equation}
        \label{p_zirochka} \lim_{t\rightarrow\infty}P\{\bmu_{\bn
        t}(X)=\bk|\bn \neq 0\}=p_\bk^*,
    \end{equation}
    and the generating function $h^*(s)=\sum_{\bk\in\natvec}p^*_\bk s^\bk$ is not depending on $\bn$ and satisfies the relationships
    %verkhnje rivnjanna 12 labels
    \begin{eqnarray}
        \label{phi_zirochka} &&1-h^*(h(\cdot))=\delta(1-h^*(s)),\\
        \nonumber &&h^*(0,\ldots,0,\ldots)=0\;,h^*(1,\ldots,1,\ldots)=1;
    \end{eqnarray}
    b) distribution $p^*_\bk$ has positive expectation
    \begin{equation*}
        h^*_j(1)=\lim_{s\rightarrow1}h^*_j(s)=\sum_{\bk\in\natvec}k_j p_\bk^*,
    \end{equation*}
    where $h^*_j(s)=\frac{\partial h^*(s)}{\partial s_j}$.
\end{theorem}
It is proved by mimicking the theorem 3 on page 198 from [3] with the use of theorem \ref{teorema_vidnosh_r} for the representation of the limit of the generating function of the conditional distribution by getting result similar to one in [2].

Let us fix one more assumption
\begin{umova}
\label{u4} Let $h_{ij}(s)=\frac{\partial h_i(s)}{\partial s_j}$, then for all $j,\;1\leq j < \infty$ there exists such $i,\; 1\leq i < \infty$, that $h_{ij}(0)$ are positive.
\end{umova}
From the equality
\begin{equation*}
    h_{ij}(0)=\widehat{P}(0,\cE(i),\cE(j))=P\{\bmu_{\cE(i)
    1}(X)=\cE(j)\}
\end{equation*}
this means, that the corresponding probabilities $\widehat{P}(0,\cE(i),\cE(j))$ are
positive.

To proceed further we need following lemma
\begin{lemma}
\label{osnlema1} Under the assumptions \ref{u1}-\ref{u4}, the limit of conditional
probabilities is positive, for all $i=1,2,\ldots$
\begin{equation*}
    \lim_{t\rightarrow\infty}P\{\mu_{\bn t}(X)=\cE(i)|n\neq0\}=p^*_{\cE(i)}>0,
\end{equation*}
\end{lemma}
\begin{proof}
    The generating function $h^*(s)=\sum _{k}p^*_k s^k$  in Theorem \ref{tmaB}
    satisfies the equation (\ref{phi_zirochka}). If we replace in this equation $s$ by $h(s)$, and repeat this replacement $t$ times, we get the equality
    \begin{equation}
        \label{lema:10}
        1-h^*(h(t,s)) = \delta^t(1-h^*(s)),
    \end{equation}
    where $h(t,s)$ is $t$-th iteration of the function implied by the main differential
    equation. By differentiating (\ref{lema:10}) with respect to $s_j$ at $s=0$, we obtain
    \begin{equation}
        \label{lema:11}
        \sum_{i=1}^\infty h^*_i(h(t,0))h_{ij}(t,0)=\delta^t
        h^*_j(0)=\delta^tp^*_{e(j)}.
    \end{equation}
    As all coordinates of $h(t,0)$ converge to $1$, for $t\rightarrow\infty$, then by the theorem \ref{tmaB} we can find such $T$ and $C_1$, that $h^*_i(h(t,0))\geq C_1>0$ for $t>T$. According to the assumption \ref{u4} this implies that for all $1\leq j \leq \infty$ we can found such $i$, that $h_{ij}(t,0)>0$. For all $i_1,i_2,\dots,i_{t+1}$ holds
    \begin{equation*}
        h_{i_1i_{t+1}}(t,0)\geq \prod_{l=1}^th_{i_1}h_{i_{l+1}}(0).
    \end{equation*}
    Thus (\ref{lema:11}) implies
    \begin{equation*}
        \delta^tp^*_{e(j)}\geq C_1\sum_{i=1}^th_{ij}(t,0)>0, \forall 1\leq j \leq
    \infty.
    \end{equation*}
\end{proof}
\begin{theorem}
    By the assumption \ref{u1} the limiting extinction probabilities
    $q^{\bn}_{\br}=\lim_{t\rightarrow\infty}q^{\bn}_{\br}(t)$, $\forall \bn \notin S,\,\br\in     S$, can be written in the series representation
    \begin{equation}
        \label{tma7:12}
        q^{\bn}_{\br}=\sum_{l=1}^\infty \sum_{\balpha\in S}c_{\balpha
        \br}\widehat{P}(l,n,\balpha),
    \end{equation}
    where $c_{\balpha \br}=\lim_{t\rightarrow\infty}c_{\balpha \br}(t,l)=\delta_{\balpha \br}-\sum_{u=1}^\infty \widetilde{P}(u,\balpha \br).$
\end{theorem}
\begin{proof}
Probabilities $q_{\br}^{\bn}(t)$ increase with $t$ and are bounded
above by $1$. Then the limit $q^{\br}_{\bn} = \lim_{t\rightarrow\infty}q^{\br}_{\bn}(t)$ exists.

We can pass to the limits on the left and on the right hand sides of the formula (\ref{tma1:3}), when $t\rightarrow\infty$, as for all $\balpha,\br\in S$ holds that $\widetilde{P}(l,\balpha,\br)\leq \widehat{P}(l,\balpha,\br)$
and Chebyshev inequality and assumption \ref{u3} imply that
\begin{eqnarray*}
    \widehat{P}(l,\balpha,\br) &\leq& P\left\{\sum_{j=1}^\infty
    \mu_{\balpha l}(\cE(j))\geq1\right\}\\
    &\leq& \sum_{j=1}^\infty \mE\left\{\mu_{\balpha l}(\cE(j))\right\}\\
    &=& \sum_{i=1}^\infty \alpha_i \sum_{j=1}^\infty
    d_{ij}\delta^l(1+o(1)).
\end{eqnarray*}
This means that series $\sum_l \widetilde{P}(l,\balpha,\br)$ and $\sum_l
\widehat{P}(l,\balpha, \br)$ converge to each other. This implies (\ref{tma7:12}).
\end{proof}
As in [2] let us consider the asymptotic behavior of $q^{\bn}_{\br}$ for $\overline{\bn}\rightarrow\infty$.
\begin{theorem}
    \label{osn_tma3}
    Let assumptions \ref{u1}-\ref{u3} are fulfilled and
    $\lim_{\overline{\bn}\rightarrow\infty}(n_i/\overline{\bn})=a_i$,
    where $a=(a_1,a_2,\ldots)$. In this case for $\br\in S$ and $\overline{\bn}\rightarrow\infty$
    \begin{equation}
        \label{3tma3:13}
        q_{\br}^{\bn}-H(\log_\delta \overline{\bn})\rightarrow 0,
    \end{equation}
    where $H(x)$ is a cyclic function with period $1$, defined through the following equalities
    \begin{eqnarray*}
        H(x)&=&\sum_{j=1}^{r_0}c_jH_j(x),\\
        H_j(x)&=&\sum_{L=-\infty}^\infty\delta^{j(L+x)}e^{-(\ba,K)\delta^{L+x}},
    \end{eqnarray*}
    where $(\ba,K)=\sum_{i=1}^\infty a_iK_i$, $K_i$
    as in (\ref{osnovna_tma_8}), $r_0=\max\{\overline{\br}=r_1+r_2+\ldots:\br\in
    S\}$. Constants $c_j=c_j(\br,\ba,p^*)$ depend on $\br$, $\ba$ and
    the limit distribution $p^*=\{p_k^*\}$ which is defined in lemma \ref{osnlema1}.
\end{theorem}
\begin{proof}
    Let $\theta(l)=(\theta_1(l),\theta_2(l),\ldots)$ be a random vector, which components $\theta_i(l)$ are equal to the number of particles of type $i$ which give an offspring to the $l$-th generation. Thus we can write, that for all $\balpha\in S,\; l\geq 1$ ³ $\bn\notin S$, we have
    %labels druhe 16
    \begin{eqnarray}
        \nonumber
        \widehat{P}(l,\bn,\balpha)&=&\sum_{\{\bbeta:1\leq\overline{\bbeta}\leq\overline{\balpha}\}}P\big\{ \bmu_{\bn,l}(X)=\balpha,\;\theta(0,l)=\bbeta
        \big\}\\
        \label{3tma3:14}
        &=&
        \sum_{\{\bbeta:1\leq\overline{\bbeta}\leq\overline{\balpha}\}}P\big\{
        \theta(0,l)=\bbeta \big\}P\{ \bmu_{\bbeta l}(X)=\balpha\,|\,\theta(0,l)=\bbeta
        \}.
    \end{eqnarray}
    Under the assumptions of the theorem \ref{osn_tma3}
    \begin{eqnarray}
        \nonumber
        P\big\{\theta(0,l)=\bbeta \big\} &=&
            \prod_{i=1}^\infty
            {n_i \choose \beta_i}
            \big(
                \widehat{P}(l,\cE(i),0)
            \big)
            ^{n_i-\beta_i}
            \big(
                1-\widehat{P}(l,\cE(i),0)
            \big)
            ^{\beta_i}\\
        \label{3tma3:15}
        &=&
            \overline{\bn}^{\overline{\bbeta}}\frac{\balpha^{\bbeta}}{\bbeta!}K^{\bbeta}\delta^{l\overline{\bbeta}}
            e^{-(\ba,K)\overline{\bn}\delta^l(1+1+o(1))}
            \big(
                1+o(1)
            \big),
    \end{eqnarray}
    and the probability, not depending on $\bn$
    \begin{eqnarray}
        \nonumber
        &&P \big\{ \bmu_{\bbeta l}(X)=\balpha \,|\,\theta(0,l)=\bbeta
        \big\}\\
        \nonumber
        &&=\sum_{\{\balpha^{(jk)}\}} \prod_{k=1}^\infty
        \prod_{j=1}^{\beta_k} P
        \big\{
            \bmu_{\cE(k),l}^{jk}(X)=\balpha^{(jk)}\,|\, \cE(k)\neq 0
        \big\}\\
        \label{3tma3:16}
        &&\longrightarrow \sum_{\{\balpha^{(jk)}\}} \prod_{k=1}^\infty
        \prod_{j=1}^{\beta_k}
        p^*_{\balpha^{(jk)}}\;\mbox{ïðè}\;l\rightarrow\infty,
    \end{eqnarray}
    where $\bmu_{\cE(k),l}^{jk}(X)$ are branching processes, whith the same distribution as
    $\bmu_{\cE(k),l}(X)$. The summation in $\sum_{\{\balpha^{(jk)}\}}$ is done over all such $\balpha^{(jk)}$, which
    $\sum_{k=1}^\infty\sum_{j=1}^{\beta_k}\balpha^{(jk)}=\alpha$. The statements in (\ref{3tma3:14})-(\ref{3tma3:16}) imply, that the general component of the series (\ref{tma7:12}) for $\overline{\bn}\rightarrow\infty,\;l\rightarrow\infty$ can be written in the form
    \begin{eqnarray}
        \nonumber
        (1+o(1))\sum_{\balpha\in
        S}\sum_{\{\bbeta:1\leq\overline{\bbeta}\leq\overline{\balpha}\}}
        g(\balpha,\bbeta)\sum_{\overline{\bbeta}}^{r_0} \delta^{(l+\log_\delta
        \overline{\bn})\overline{\bbeta}}\\
        \label{3tma3:17}
        \times \exp\big\{ -(\ba,K)\delta^{l+\log_\delta\overline{\bn}}(1+o(1))
        \big\},
    \end{eqnarray}
    where $g(\balpha,\bbeta)$ is an independent of $\bn$ and $l$ function.
    It is easy to see that in formula (\ref{tma7:12}) for $\overline{\bn}\rightarrow\infty$ each component of series with any $l\geq 1$ converges to zero.

    Let us choose $L_1<L_2$ in such way that sums
    \begin{equation}
        \label{3tma3:18}
        \sum_{L=L_2}^\infty
        \delta^{\overline{\bbeta}L}e^{-(\ba,K)\delta^L}
        \;\mbox{and}\;\sum_{L=-\infty}^{L_2}
        \delta^{\overline{\bbeta}L}e^{-(\ba,K)\delta^L}
    \end{equation}
    are small. We set $l_i+\log_{\delta}\overline{\bn}=L_i+x_{i\overline{\bn}},$ for $i=1,2$,     where $0\leq x_{i\overline{\bn}}\leq1$. (\ref{3tma3:17}) and
    (\ref{3tma3:18}) imply, that we can choose such $L_1,\,L_2$ and
    $n_0$, that tails of the sum in (\ref{tma7:12}), bounded from $1$
    to $l_1$ and from $l_2$ to infinity, are less then
    $\varepsilon/2$, where $\varepsilon>0$ is small. Elements of the series
    (\ref{tma7:12}) with $l_1<l<l_2$ can be replaced by a limited expressions
    (\ref{3tma3:17}) for $\overline{\bn}\rightarrow\infty$
    as well as for $l\rightarrow\infty$. The number of summands in the sum
    $\sum_{l=l_1+1}^{l_2-1}$ in expression (\ref{tma7:12}) is finite
    $l_2-l_1-1=L_2-L_1-1$. This means that $n_0$ can be chosen in such a way, that for all $\bn>n_0$ the approximation error will be also less than
    $\varepsilon/2$. This implies the statement of the theorem, while $\varepsilon
    >0$ is any real number.
\end{proof}
From the theorem it cannot be concluded directly, whether the coefficients $c_j$ in the
formula (\ref{3tma3:13}) are such, that $H(x)>0$. For this we introduce the next lemma.
\begin{lemma}
    Under assumptions \ref{u1}-\ref{u4}, there exists such a constant $\Theta>0$, that
    for some number $n_0$
    \begin{equation*}
        q_{\br}^{\bn} > \Theta {\text{, for }} \forall \bn {\text{ with }} \overline{\bn}\geq n_0 {\text{ and }}\forall \br\in S.
    \end{equation*}
\end{lemma}
\begin{proof}
    As for any $t$, $q_{\br}^{\bn}=\lim_{t\rightarrow\infty}q_{\br}^{\bn}(t)\geq
    q_{\br}^{\bn}(t)$, it is enough to prove, that the inequality  $q_{\br}^{\bn}(t)\geq
    \Theta>0$ holds for any large enough $t$, for all $\br\in S$ and $\bn$ from $\overline{\bn}\geq n_0$. Let us use the upper defined random vector $\theta(0,t)$ and
    introduce one more random vector $\theta'_i(t-1)=(\theta'_1(t-1),\theta'_2(t-1),\ldots)$, where $\theta'_i(t-1)$ is the number of starting particles of $i$-th type, with nonempty offspring set in the $(t-1)$-th generation, but empty in the $t$-th generation. For $\bn=(n_1,n_2,\ldots)$, $n_1\geq r_0+1$, where $r_0=\max_{\br\in S}\overline{\br}$, we use the inequality
    \begin{eqnarray}
        \nonumber
        q_{\br}^{\bn}(t) &=& P\big\{ \bxi_{\bn t}(X)=\br \big\} \\
        \label{2lma2:19}
        &\geq& P\big\{ \bmu_{\bn t}(X)=\br,\;\theta'(t-1)=(r_0+1-\overline{\br}),\;\theta(0,t)=\overline{\br}\cE(1)
        \big\}.
    \end{eqnarray}
    The right side of (\ref{2lma2:19}) we write as a product
    of ${\cP}_1(\bn,t){\cP}_2(t)$, where
    $\cP_1(\bn,t)=P\big\{ \theta'(t-1)=(r_0+1-\overline{\br})\cE(1),\;\theta(0,t)=\overline{\br}\cE(1) \big\}$
    and depends on $\bn$ and $t$, but
    $\cP_2(t)=P\big\{ \bmu_{\bn t}(X)=\br \,|\, \theta'(t-1)=(r_0+1-\overline{\br})\cE(1),\; \theta(0,t)=\overline{\br}\cE(1) \big\}$
    and depends only on $t$. From the definition of the random vectors
    $\theta(0,t)$ and $\theta'(t-1)$ we have, that
    \begin{eqnarray}
        \nonumber
        \cP_1(\bn,t) &=&
        \frac{n_1!}{(n_1-r_0-1)!(r_0+1-\overline{\br})!\overline{\br}}\big[
        \widehat{P}(t-1,\cE(1),0)\big]^{n_1-r_0-1} \\
        \nonumber
        &\times& \big(
        1-\widehat{P}(t-1,\cE(1),0)\big)^{r_0+1-\overline{\br}}\big(
        1-\widehat{P}(t-1,\cE(1),0)\big)^{\overline{\br}} \\
        &\times& \prod_{i=1}^\infty
        \label{2lma2:20}
        [\widehat{P}(1,\cE(1),0)^{n_i}]P\big\{ \bmu'_{r_0+1-\overline{\br} \; t}(X)=0\,|\, r_0+1-\overline{\br}\neq
        0\big\};\\
        \label{2lma2:21}
        \cP_2(t)&=&\prod_{k=1}^\infty \prod_{j=1}^{r_k} P\big\{
        \bmu^{(jk)}_{\cE(1)t}(X)=\cE(k)\,|\,\cE(k)\neq0
        \big\}.
    \end{eqnarray}
    Here $\bmu,\;\bmu',\;\bmu^{(jk)}$ are branching processes, whose evolution is defined by a generating function $h(s)=(h_1(s),h_2(s),\ldots)$. Setting $t\rightarrow\infty$, in such a way, that $\overline{\bn}\delta^t\rightarrow V>0$ for $\overline{\bn}\rightarrow\infty$, we get in the right side of the equality (\ref{2lma2:20}) a positive constant multiplied by a conditional probability, which stays at the end of formula. Using the limiting relationship $P\big\{ \bmu'_t(X)=\bk\,|\, \bk\neq0 \big\} \rightarrow p^*_{\bk}$, of the theorem \ref{tmaB} and the equality \[\sum_{\bk \in \natvec} p^*_{\bk} \big(\widehat{P}^{k_1}(1,\cE(1),0)\widehat{P}^{k_2}(1,\cE(2),0)\cdots\big)=h^*(h(0))\]
    we have that this conditional probability is equal in limit to $h^*(h(0))$.
    Expression (\ref{2lma2:21}) does not depend on $\bn$ and is equal to the product
    $\prod_{i=1}^\infty[p^*_{\cE(i)}]^{r_i}$, for $t\rightarrow\infty$. From lemma \ref{osnlema1} this product is positive. That completes the proof.
\end{proof}


\begin{thebibliography}{777}
    \bibitem [1]{ll2}Yeleyko Ya. I. (1994). Asymptotic Analysis and transition events in the matrixvalued random evolutions, branching processes and procceses with the Markov properties, Habilitation, Lviv, (in Ukrainian).
    \bibitem [2]{ll1}Sevastyanov B. A. (1999). Asymptotic Behavior of the Extinction Probabilities for Stopped Branching Processes  // Theory of Probability and its Applications \textbf{43}, pp. 315-322.
    \bibitem [3]{ll3}Sevastyanov B. A. (1971). Branching Processes. Moscow, Nauka.
    \bibitem [4]{ll3}Harris T. E. (2002). The Theory of Branching Processes. Courier Dover Publications.
\end{thebibliography}
\end{document}